\documentclass[12pt]{amsart}

\usepackage{amssymb}

\usepackage{enumerate}

\newtheorem{thm}{Theorem}[section]
\newtheorem{cor}[thm]{Corollary}
\newtheorem{lem}[thm]{Lemma}
\newtheorem{prop}[thm]{Proposition}

\newtheorem{mainthm}[thm]{Main Theorem}

\theoremstyle{definition}
\newtheorem{defin}[thm]{Definition}

\newtheorem{exa}[thm]{Example}

\newtheorem*{xrem}{Remark}

\numberwithin{equation}{section}

\newcommand{\N}{\mathbb{N}}

\newcommand{\R}{\mathbb{R}}

\newcommand{\fB}{\mathfrak{B}}

\newcommand{\fC}{\mathfrak{C}}

\newcommand{\fS}{\mathfrak{S}}

\newcommand{\fT}{\mathfrak{T}}

\newcommand{\ff}{\mathrm{f}}

\newcommand{\cI}{\mathcal{I}}

\newcommand{\cM}{\mathcal{M}}

\newcommand{\cP}{\mathcal{P}}

\newcommand{\tmu}{\tilde{\mu}}
\newcommand{\tnu}{\tilde{\nu}}

\begin{document}


\baselineskip=17pt


\title{A note on product of measures}

\author{Grzegorz Andrzejczak}
\address{Institute of Mathematics\\ 
Lodz University of Technology, Poland}
\email{grzegorz.andrzejczak@p.lodz.pl}

\date{}


\renewcommand{\thefootnote}{}



\renewcommand{\thefootnote}{\arabic{footnote}}
\setcounter{footnote}{0}


\begin{abstract}
A slight modification to Halmos' definition of product of measures yields a uniquely characterized associative product. The operation applies to arbitrary (not necessarily $\sigma-$finite) measures and is consistent with the Fubini--Tonelli theorem.
\end{abstract}
\maketitle

\section{Introduction}


In elementary context, it is a generally accepted convention that any reasonable product $\mu\otimes\nu$ of two measures defined in measure spaces, say $(S,\fS)$ and $(T,\fT),$ takes advantage of some $\sigma-$finitness assumptions. One of the most general elementary definitions, proposed by Halmos in {\cite{Halmos}}, assumes that $\fS$ and $\fT$ are $\sigma-$rings, and requires $\mu$ and $\nu$ to be $\sigma-$additive $\sigma-$finite measures on the rings. Keeping in mind that there exist various refined and elaborated generalizatons of \emph{the product of measures} (see e.g. \cite{Fremlin}), I would like to note that the Halmos' approach can be easily and succesfully applied to \emph{arbitrary measures}. While being still important, $\sigma-$finitness is no longer an assumption -- it becomes the border between a computable (numeric) part and a declarative, purely infinite part of any measure.

\section{The product}
\begin{defin} By a \emph{measurable space} we shall mean any pair $(S,\fS)$ composed of a nonempty set $S$ and a $\sigma-$ring $\fS$ of subsets of $S.$ An extended $\sigma-$additive real function $\mu\colon\fS\to[0,\infty]$ will be called a \emph{measure} in the measurable space if $\mu(\emptyset)=0.$ A triple $(S,\fS,\mu)$ is a \emph{measure space} if $\mu$ is a measure in $(S,\fS).$ 
\end{defin}
We recall that a family $\fC$ of sets is a $\sigma-$ring if it is closed under countable unions and if $A\setminus B\in\fC$ whenever $A,B\in\fC.$ 

Any measure $\mu$ in $(S,\fS)$ distinguishes ihe family  of sets of finite measure, $\fS_\mu^\ff=\{A\in\fS;\;\mu(A)<\infty\},$ as well as the $\sigma-$ring $\fS_\mu^\sigma$ of $\sigma-$finite sets. Precisely, $\fS_\mu^\sigma$ consists of all the unions of countable subsets of $\fS_\mu^\ff,$ and is the smallest $\sigma-$ring containing all sets of finite measure.
 
\begin{defin} For an arbitrary measure space $(S,\fS,\mu)$ the restriction of $\mu$ to the $\sigma-$ring $\fS_\mu^\sigma$ will be called the \emph{$\sigma-$finite component} of the measure and will be denoted by $\mu^\sigma.$ The corresponding triple $(S,\fS_\mu^\sigma,\mu^\sigma)$ is the \emph{$\sigma-$finite component} of the measure space.  
\end{defin}
The component is a $\sigma-$finite measure in terminology used by Halmos. In fact, Halmos calls a measure $\sigma-$finite if and only if it its domain is generated by sets of finite measure. 

The pair composed of $\fS$ and $\fS'=\fS_\mu^\sigma$ has an important property
\begin{align}\label{def:simple:ext}
\fS'\subset\fS\quad\text{and}\quad
\forall_{A\in\fS}\forall_{B\in\fS'}\; A\subset B\Longrightarrow A\in\fS'.
\end{align}

\begin{prop}\label{prop:simple:ext} Given any pair $(\fS,\fS')$ of $\sigma-$rings having \emph{the simple extension property} \eqref{def:simple:ext}, for every measure $\mu'\colon\fS'\to[0,\infty]$ the  extension $\mu$ of $\mu'$ such that $\mu(A)=\infty,$ for $A\in\fS\setminus\fS',$ is a measure. 
\end{prop}
\begin{proof}  For any equality of the form $C=\bigcup_n C_n,$ where $C_n\in\fS$ as $n\in\N,$  one has $C\in\fS'$ if and only if every summand is in $\fS'.$ 
\end{proof}

For any family of sets $\fC$ we shall denote by $\sigma(\fC)$ the $\sigma-$ring generated by the family, ie. the smallest $\sigma-$ring containing $\fC.$
Obviously, one has $\fS_\mu^\sigma=\sigma(\fS_\mu^\ff).$
An analogous notion of $\sigma-$algebra is relative and depends on \emph{the space} that is a fixed set, say $S,$ such that  $\bigcup\fC\subset S.$ The $\sigma-$ring is a $\sigma-$algebra if and only if containes the space $S.$  
Modifying a classical notation and making it more precise we set
\[ \sigma_S(\fC) := \sigma(\fC\cup\{S\})\]
for the $\sigma-$algebra generated by $\fC.$
Any function $f\colon S\to\R$ is \emph{measurable} with respect to a $\sigma-$ring $\fS$ if $f$ is $\sigma_S(\fS)-$measurable and $\{x;\;f(x)\!\neq\! 0\}\!\in\!\fS.$ Obviously, $\sigma_S(\fS)=\fS\cup\{S\setminus A;\;A\in\fS\}.$

We recall that for any two $\sigma-$rings $\fS$ and $\fT$ the $\sigma-$ring
\[ \fS\otimes\fT = \sigma(\{A\times B;\; A\in\fS,B\in\fT\})
\]
is called the product of the $\sigma-$rings. Clearly, the product is a $\sigma-$algebra if and only if both $\fS$ and $\fT$ are $\sigma-$algebras. By the product $(S,\fS)\times(T,\fT)$ of two measurable spaces we mean the product space $S\times T$ equipped with the product $\sigma-$ring $\fS\otimes\fT.$ The product of measurable spaces is associative.

Let us consider arbitrary  measure spaces $(S,\fS,\mu)$ and $(T,\fT,\nu).$
\begin{thm}[Product of $\sigma-$finite measures, see {\cite{Halmos}}]\label{T:1}
If the $\sigma-$rings $\fS$ and $\fT$ are generated by sets of finite measure then there exists a unique measure $\mu\otimes\nu\colon\fS\otimes\fT\to[0,\infty]$ such that 
\[ (\mu\otimes\nu)(A\times B)=\mu(A)\cdot\nu(B)\quad\text{for }A\in\fS,B\in\fT. \]
\end{thm}
\begin{xrem} Although the proof presented in \cite{Halmos} uses the Lebesgue integral, there exist more direct proofs (see e.g. \cite{Doob}) which concentrate on the case when both measures are finite. By uniqueness, the construction is then extended to a consistent family of measures on arbitrary products $S'\times T'\subset S\times T$ of $\sigma-$finite measurable sets $S'\in\fS,$ $T'\in\fT.$
\end{xrem}
Without any assumption on the measures we claim what follows.
\begin{cor} There exists a unique measure $\mu\otimes\nu$ in $(S\times T,\fS\otimes\fT)$ such that
\begin{align}\label{wz:mu:nu:sigma} (\mu\otimes\nu)^\sigma = \mu^\sigma\otimes\nu^\sigma, 
\end{align}
ie. whose $\sigma-$finite component is the product of $\sigma-$finite components of $\mu$ and $\nu.$ The \emph{product of measures} $\mu\otimes\nu$ is the only measure in the product of measurable spaces which has the following two properties:
\begin{enumerate}[\upshape (i)]
\item $(\mu\otimes\nu)(A\times B) = \mu(A)\cdot\nu(B)$ for any $A\in\fS_\mu^\ff,$ $B\in\fT_\nu^\ff.$
\item The $\sigma-$ring $(\fS\otimes\fT)_{\mu\otimes\nu}^\sigma$ of all $\sigma-$finite sets in $\fS\otimes\fT$ is generated by the family 
 $\{A\times B;\; A\in\fS_\mu^\ff, B\in\fT_\nu^\ff\}.$ 
\end{enumerate}
\end{cor}
\begin{proof} According to lemma~\ref{lem:fSxfT:sigma}, the pair $(\fS\otimes\fT,\fS_\mu^\sigma\otimes\fT_\nu^\sigma)$ has the simple extension property. Thus the measure $\mu^\sigma\otimes\nu^\sigma$ is uniquely extendible to a measure in $(S\times T,\fS\otimes\fT)$ which has no more sets of finite measure. 
\end{proof}
In view of associativity of the product of $\sigma-$finite measures, equality \eqref{wz:mu:nu:sigma} gives rise to
\begin{cor} The above product of arbitrary measures is associative. \qed
\end{cor}
\begin{xrem}
The product measure $\mu\otimes\nu$ can be obtained via the Caratheo\-dory formalism, if one starts with the semi-ring of ''rectangles'' $A\times B$ and the function $A\times B\mapsto \mu(A)\cdot\nu(B),$ for $A$ and $B$ of finite measure -- as in (i). However, if at least one of sets  $A\in\fS,$   $B\in\fT$ is \emph{not $\sigma-$finite} and the other set is nonempty then
$(\mu\otimes\nu)(A\times B)=\infty,$ while the product $\mu(A)\cdot\nu(B)$ equals either $\infty$ or $0$ (according to the \emph{axiom} $\infty\cdot 0=0).$ 
\end{xrem}
\begin{exa} In the Borel measurable space $(\R,\fB_R)$ on the real line
the product of the Lebesgue measure $\ell$ and the counting measure $\delta$ is a Borel measure on $\R^2.$ Any Borel set $B\subset \R^2$ is $\sigma-$finite with respect $\ell\otimes\delta$ if and only if it is of the form $B=\bigcup_{n\in\N} A_n\times\{a_n\},$ where $A_n\in\fB_\R$ and $a_n\in\R$ for $n\in\N.$
\end{exa}
In order to deal with Lebesgue integrals, we propose the following
\begin{defin}\label{def:int1} For an arbitrary measure space $(S,\fS,\mu)$
 by a \emph{Lebesgue integral with respect to the measure} we mean the only non-negative linear functional $\int\!d\mu\colon\cI(S,\mu)\to\R,$ $f\mapsto\int\!f\,d\mu,$ where the linear space $\cI(S,\mu)$ consists of finite $\fS-$measurable real functions on $S$, and
\begin{enumerate}[\upshape (i)]
\item for any set $A\!\in\!\fS,$ one has
$ 1_A\!\in\!\cI(S,\mu)\Longleftrightarrow\mu(A)\!<\!\infty,$
and if $A\!\in\!\fS_\mu^\ff$ then $\int 1_A d\mu=\mu(A);$ 
\item for every non-decreasing sequence $(f_n)_{n\in\N}$ bounded at each point of $S,$ if the sequence $(\int\!f_n\,d\mu)_{n\in\N}$ is bounded then the pointwise limit $f=\lim_{n\to\infty}f_n$ is an element of $\cI(S,\mu),$ and one has
\[\int\!f\,d\mu=\lim_{n\to\infty}\int\!f_n\,d\mu.\]
\item if $f\in\cI(S,\mu)$ then also $|f|,\min(f,1)\in\cI(S,\mu);$
\end{enumerate}
Finite $\fS-$measurable functions which are elements of $\cI(S,\mu)$ are called \emph{$\mu-$integrable}.
\end{defin}
\begin{xrem} Properties (i)--(ii) are well-known to characterize the Lebesgue integral as the linear functional having \emph{the smallest} domain, and condition (iii) assures that the space $\cI(S,\fS)$ of integrable functions is not bigger. 
In fact, property (iii) means that the space of $\mu-$integrable functions is a Stone lattice. Together with the other properties, the Stone condition $\min(f,1)\!\in\!\cI(S,\mu)$ is equivalent to the assertion $\{x;\;f(x)\!\neq\! 0\}\!\in\!\fS_\mu^\sigma$ for $f\in\cI(S,\mu),$ and is superfluous if the measure space is $\sigma-$finite. 
  
\end{xrem}
\begin{cor}\label{cor:int:Leb} The Lebesgue integral with respect to an arbitrary measure $\mu$ equals 
$\int\!d\mu^\sigma$ i.e. integrability as well as the integral depend on the $\sigma-$finite component $\mu^\sigma$ only. \qed
\end{cor}
Given any measurable space $(S,\fS),$ let $\cM^+(S,\fS)$ stand for the cone of nonnegative extended real-valued $\fS-$measurable functions on $S.$
\begin{defin}\label{def:int2} By an \emph{extended Lebesgue integral} in an arbitrary measure space $(S,\fS,\mu)$ we mean the only non-decreasing function $\int\!d\mu\colon \cM^+(S,\fS)\to[0,\infty],$ equal to the integral $\int\!d\mu\colon\cI(T,\mu)\to\R$ on non-negative finite integrable functions, and such that
\begin{enumerate}[\upshape (i)]
\item for any set $A\!\in\!\fS,$ one has
 $\int 1_A d\mu=\mu(A);$ 
\item for every non-decreasing sequence $(f_n)_{n\in\N}$ in $\cM^+(S,\fS),$ the following equality
\[\int\!\lim_{n\to\infty}f_n\,d\mu=\lim_{n\to\infty}\int\!f_n\,d\mu\]
holds true.
\end{enumerate}
An extended real-valued $\fS-$measurable function $f$ on $S$ is called \emph{$\mu-$inte\-grable} if $\int\!|f|\,d\mu<\infty.$ 
\end{defin}
As a complement to corollary~\ref{cor:int:Leb} we get
\begin{cor}\label{cor:int2:Leb} For any $f\in\cM^+(S,\fS)$ the function is integrable if and only if there exists a finite integrable $g\in\cI(T,\mu)$ such that $f=g$ \emph{almost everywhere}, i.e. $\mu(\{x;\; f(x)\neq g(x)\})=0.$  

Every $\mu-$integrable function is $\fS_\mu^\sigma-$measurable  and $\mu^\sigma-$integrable. \qed
\end{cor}
Classical expositions of the Lebesgue integral in $(S,\fS,\mu)$ take advantage of a Daniell--Stone formalism and start from an extension of the assignment $1_A\mapsto \mu(A),$ for $A\in\fS_\mu^\ff,$ to a unique linear functional $\tmu\colon\cP_\mu\to\R$ \emph{associated with the measure}. The domain $\cP_\mu$ is (algebraically) generated by the characteristic functions, and consists of -- so called -- simple functions. The respective Daniell--Stone integral $\int\!d\tmu$ is well-known to be a completion of the Lebesgue integral  
$\int\!d\mu.$

In the case of two arbitrary measure spaces, $(S,\fS,\mu)$ and $(T,\fT,\nu),$ the tensor product $\cP_\mu\!\otimes\!\cP_\nu$ is naturally isomorphic to a linear subspace of ${\fS\!\otimes\!\fT}-$measurable functions on $S\times T,$ and the respective associated functionals yield a functional $\tmu\otimes\tnu\colon \cP_\mu\otimes\cP_\nu \to \R$ such that
\[(\tmu\otimes\tnu)(f)=\tmu\big(s\mapsto\tnu(f(s,\cdot))\big)= \tnu\big(t\mapsto\tmu(f(\cdot,t))\big),
\]
for any $f\in\cP_\mu\otimes\cP_\nu.$  
Turning back to the examined product of measures, 
we are now about to formulate and prove
\begin{mainthm} Let $(S,\fS,\mu)$ and $(T,\fT,\nu)$ be any measure spaces.
\begin{enumerate}[\upshape (i)]
\item The Lebesgue integral $\int\!d(\mu\otimes\nu)$ is equal to the Daniell--Stone integral $\int\!d(\tmu\otimes\tnu)$ -- restricted to $\fS\otimes\fT-$measurable functions.
\item  \emph{Fubini:} For any $\mu\otimes\nu-$integrable function $f\colon S\!\times\! T\!\to\![-\infty,\infty],$ one~has  
\begin{align}\label{wz:int:fubini}
 \int\! f\, d(\mu\otimes\nu) = \int\!\left(s\mapsto\int\! f(s,\cdot)\,d\nu\right)d\mu \\ \notag
= \int\!\left(t\mapsto\int\! f(\cdot,t)\,d\mu\right)d\nu,
\end{align}
where the integrands on the right are integrable for almost every $s$ and $t,$ respectively.
\item \emph{Tonelli:} Equalities \eqref{wz:int:fubini} remain valid if $f\in\cM^+(S\times T,\fS\otimes\fT)$ is $\fS_\mu^\sigma\otimes\fT_\nu^\sigma-$measurable, i.e. such that the set $\{x;\;f(x)\neq 0\}$ is $\sigma-$finite.  Finite value of any of the three sides of \eqref{wz:int:fubini} assures then $(\mu\otimes\nu)$--integrabi\-lity of $f.$
\end{enumerate}
\end{mainthm}
\begin{proof} (i) The space $\cP_\mu\otimes\cP_\nu$ is a Stone lattice, so the Stone theorem (see e.g. \cite{Stroock}) assures that the Daniell--Stone integral $\int\!d(\tmu\otimes\tnu)$ is equal to the Lebesgue integral with respect to a measure, say $\lambda,$ and the corresponding $\sigma-$finite sets form the $\sigma-$ring $\fS_\mu^\sigma\otimes\fT_\nu^\sigma.$ Since the product $\mu^\sigma\otimes\nu^\sigma$ and the $\sigma-$finite component $\lambda^\sigma$
of $\lambda$ are both defined on the same $\sigma-$ring and are equal on the $\pi-$system $\{A\times B;\; A\in \fS_\mu^\ff, B\in \fT_\nu^\ff\},$ they are equal -- and so $\lambda=\mu\otimes\nu.$ 

Assertions (ii)--(iii) follow from corollaries~\ref{cor:int:Leb}--\ref{cor:int2:Leb} and the classical $\sigma$--finite variant of the Fubini--Tonelli theorem.
\end{proof}
\section{Technical lemmas}
For any family of sets $\fC$ and a set $S$ let $\fC|S:=\{A\cap S;\;A\in\fC\}$ denote a \emph{restriction} of the family (to $S$). Any restriction of a $\sigma-$ring remains a $\sigma-$ring. Basic properties of the operation are recalled in 
\begin{lem}\label{lem:restrict} \begin{enumerate}[\upshape (i)] 
\item $\sigma(\fC|S)=\sigma(\fC)|S.$
\item For any $n\in\N$ and sets $S_1,\ldots,S_n,$ \[\bigotimes_{i\leq n}(\fS_i|S_i)=\big(\bigotimes\fS_i\big)|{S_1\times\cdots\times S_n}\]
whenever $\fS_i,i\leq n,$ are arbitrary $\sigma-$rings. \qed 
\end{enumerate}
\end{lem}
\begin{lem}\label{lem:fSxfT:sigma}
Let $(S_i,\fS_i,\mu_i),i\leq n,$ be any finite sequence of measure spaces. Then one has
\begin{align}\label{wz:mu:SxT:sigma}
  \bigotimes_{i\leq n}\fS_i^\sigma = \sigma\{A_1\!\times\!\cdots\!\times\! A_n;\; \forall_{i\leq n}\,\mu_i(A_i)\!<\!\infty \},
\end{align}
where $\fS^\sigma_i:=(\fS_i)_{\mu_i}^\sigma$ for $i\leq n.$ Furthermore, the  $\sigma-$ring \eqref{wz:mu:SxT:sigma} is composed of all the measurable sets $C\in\bigotimes_{i\leq n}\fS_i$ such that 
\begin{align}\label{wz:mu:S_i:sigma}
C\subset S'_1\times\!\cdots\!\times S'_n,\quad \text{for some }
S'_i\in\fS_i^\sigma,\, i=1,\ldots, n.
\end{align}
\end{lem}
\begin{proof} Equality \eqref{wz:mu:SxT:sigma} is a simple consequence of the notion of $\sigma-$finitness.  
The family of sets $C$ satisfying (\ref{wz:mu:S_i:sigma}) is a $\sigma-$ring, and thus contains the product $\bigotimes_{i\leq n}\fS_i^\sigma.$ In order to prove the reverse inclusion we consider any $\sigma-$finite sets $S'_i\in\fS_i^\sigma, i\leq n,$ and an arbitrary $\bigoplus_{i\leq n}\fS_i-$measurable subset $C\subset S'_1\!\times\!\cdots\!\times\! S'_n.$  Such a $C$ is an element of a $\sigma-$ring
\[\begin{array} {rl}
(\bigotimes_{i\leq n}\fS_i)|{S'_1\!\times\!\cdots\!\times\! S'_n} &= \bigotimes_{i\leq n}(\fS_i|{S'_i})\\
 &=\sigma\{A_1\!\times\!\cdots\!\times\! A_n;\;\forall_{i\leq n}\,A_i\in\fS_i|{S'_i}\}\\
&\subset \sigma\{A_1\!\times\!\cdots\!\times\! A_n;\;\forall_{i\leq n}\,A_i\in\fS_i^\sigma\}=\bigotimes_{i\leq n}\fS_i^\sigma. 
\end{array} \]
\end{proof}



\begin{thebibliography}{HD}




\normalsize
\baselineskip=17pt


\bibitem[D] {Doob} J.L. Doob, \emph{Measure Theory}, Springer, Berlin, 1993.
\bibitem[H] {Halmos} P.R. Halmos, \emph{Measure Theory}, Springer, Berlin, 1974.
\bibitem[F] {Fremlin} D.H. Fremlin, \emph{Measure Theory}, Vol. 2, Torres Fremlin, Colchester, 2001.
\bibitem[S] {Stroock} D.W. Stroock, \emph{Essentials of Integration Theory for Analysis}, Springer, New York, 2011.
\end{thebibliography}
\end{document}